\documentclass[12pt]{amsart}
\usepackage{amscd,amssymb,amsfonts}
\usepackage{graphicx}
\usepackage[mathscr]{eucal}

\newtheorem{thm}{Theorem}[section]
\newtheorem{lemma}[thm]{Lemma}
\newtheorem{conj}[thm]{Conjecture}

\theoremstyle{definition}
\newtheorem{defn}[thm]{Definition}

\newtheorem{rem}[thm]{Remark}

\newcommand{\R}{{\mathbb R}}

\newcommand{\E}[1]{\mathsf E(#1)}
\newcommand{\SH}[2]{\mathsf{SH}_{#1}(#2)}
\newcommand{\EK}[2]{\mathsf E_{#1}(#2)}

\newcommand{\D}[2]{\Delta_{#1}(#2)}

\newcommand{\HP}{H^{2}}

\newcommand{\BD}{\HP\times\HP}

\newcommand{\HD}{d_H}
\newcommand{\HDS}{{\HD}^2}
\newcommand{\BHD}{\rho}
\newcommand{\BHDS}{\rho^2}

\newcommand{\PSL}{{\operatorname{PSL}(2,\R)}}

\newcommand{\PGL}{{\operatorname{PGL}(2,\R)}}

\newcommand{\isom}[1]{\operatorname{Isom}(#1)}
\newcommand{\isoh}{\isom{\HP}}
\newcommand{\isob}{\isom{\BD}}

\renewcommand{\H}{\mathcal H}

\newcommand{\vz}{\mathsf{z}}
\newcommand{\vw}{\mathsf{w}}

\newcommand{\vy}{\mathsf{y}}
\newcommand{\vx}{\mathsf{x}}

\newcommand{\bg}{\gamma}

\newcommand{\bi}{\iota}

\newcommand{\BG}{\Gamma}

\begin{document}

\title{Equidistant hypersurfaces of the bidisk}

\author[Charette]{Virginie Charette}
  \address{D\'epartement de math\'ematiques\\ Universit\'e de Sherbrooke\\
  Sherbrooke, Quebec, Canada}
  \email{v.charette@usherbrooke.ca}

\author[Drumm]{Todd A.\ Drumm}
\address{Department of Mathematics\\ Howard University\\
     Washington, DC 20059  USA }
    \email{tdrumm@howard.edu}

\author[Lareau-Dussault]{Rosemonde Lareau-Dussault}
  \address{D\'epartement de math\'ematiques\\ Universit\'e de Sherbrooke\\
  Sherbrooke, Quebec, Canada}
  \email{r.lareau@usherbrooke.ca}

\begin{abstract}
The following are notes on the geometry of the bidisk,
$\BD$. In particular, we examine the properties of
equidistant surfaces in the bidisk.
\end{abstract}

\thanks{Charette gratefully acknowledges partial support from the
  Natural Sciences and Engineering Research Council of Canada.
 }

\maketitle


\section{Introduction}

The {\em bidisk} is the product of two copies of the hyperbolic plane.  While
the bidisk has been often mentioned in the literature, notably as a basic example in the theory of
symmetric spaces, it would seem that few papers are devoted to its
geometry in the strictest sense.  In this vein we might for instance point out the
paper by Eskin and Farb~\cite{EF}, which considered quasi-flats in the bidisk.

Our interest lies in fundamental polyhedra for groups of isometries that act 
properly discontinuously on the bidisk.  One question is, how many sides can a 
Dirichlet domain for a cyclic group have?  The motivation for this question arises 
from related work of J{\o}rgensen  on hyperbolic 3-space \cite{J}  (see Drumm-Poritz \cite{DP} 
for an interactive treatment), and by Phillips's work on complex hyperbolic space \cite{Ph}. 
Of related interest is a paper by Ehrlich and Im Hof~\cite{EI}, 
which describes some basic properties of Dirichlet domains in Riemannian manifolds without conjugate points.  

The present paper focuses on cyclic groups generated by the product of two hyperbolic 
isometries.  We prove that a Dirichlet domain with basepoint in the invariant flat 
necessarily has two faces.  However, if one chooses a point outside that flat, the 
Dirichlet domain need not be two-faced.

In the hyperbolic plane, every Dirichlet domain for a cyclic Fuchsian 
group is two-faced.  The same is not true, however, in hyperbolic 3-space -- as a matter of 
fact, there is no bound on the possible number of faces~\cite{J,DP}. 

The main part of the present paper is devoted to the description of {\em
equidistant hypersurfaces}, which bound Dirichlet domains.  These
three-dimensional manifolds display rich geometrical properties.  For instance
they are not totally geodesic, thus complicating the analysis of their
intersections.  This is reminiscent of the situation for {\em bisectors} in the
complex hyperbolic plane.  (See for example~\cite{Go99}.)  On the other hand,
equidistant hypersurfaces admit a foliation by products of {\em square
hyperbolae}, which are curves in the hyperbolic plane.  As we will show,
understanding how equidistant hypersurfaces intersect boils down to
understanding how square hyperbolae intersect.

The paper is organized as follows: \S\ref{sec:isom} introduces the bidisk and
describes its isometries; \S\ref{sec:ehs} describes equidistant hypersurfaces
and their foliation by products of square hyperbolae; finally, in
\S\ref{sec:Dirichlet}, after defining Dirichlet domains, we prove the main
theorem about two-faced domains.  We close with a discussion about Dirichlet domains centered
about points which do not lie on an invariant axis.

\subsection*{Acknowledgements.}  The authors thank Bill Goldman, John Parker and Deane Yang for several enlightening discussions.  We would also like 
to thank the anonymous referee for several helpful suggestions.

\section{The bidisk and its isometries}\label{sec:isom}

Given a Riemannian space $X$, denote its isometry group by $\isom{X}$.  Since
the isotropy subgroup of any point in $X$ is compact, a discrete subgroup
$G<\isom{X}$ acts properly discontinuously on $X$. 

Let $\HP$ denote the hyperbolic plane.  As its geodesics are isometric to Euclidean straight lines,
we will call them {\em straight lines} or simply {\em lines}.  In the upper half-plane model,   
$\isoh=\PGL$ and the subgroup of orientation preserving isometries is $\PSL$. 
Recall that an isometry $g\neq id \in\PSL$ belongs to one of three types:
\begin{itemize}
\item $g$ is {\em hyperbolic} if it fixes two points on the ideal boundary of
$\HP$;
\item $g$ is {\em parabolic} if it fixes a single point on the ideal boundary of
$\HP$;
\item $g$ is {\em elliptic} if it fixes a single point in $\HP$.
\end{itemize}
If $g$ is hyperbolic, it admits an invariant axis, namely, the straight line determined by  
its two fixed points.  The invariant curve for parabolic $g$ is a horocycle.

The distance between a pair of points $z,w\in\HP$  can be written as follows:
\begin{equation*}
\HD(z,w)=2\tanh^{-1}\frac{\mid z-w\mid}{\mid z-\overline{w}\mid}
\end{equation*}

\begin{rem}\label{rem:equiplane}
If $g\in\isoh$ is either hyperbolic or parabolic, then for any $z\in\HP$, the
locus of points at equal distance to $z$ and $g(z)$ is a straight line, called the {\em equidistant line}, that is
perpendicular to the $g$-invariant curve.  Moreover, the equidistant line
between $z$ and $g(z)$ is disjoint from the equidistant line between $g(z)$ and
$g^2(z)$.
\end{rem}

The bidisk is $\BD$.  Endow the bidisk with the standard (product) Riemannian
metric:
\begin{equation*}
\BHD(\vz,\vw)=\sqrt{\HD(z_1,w_1)^2+\HD(z_2,w_2)^2}
\end{equation*}

\begin{defn}
For any $z\in\HP$, a surface $\HP\times\{z\}$ will be called a {\em horizontal
plane} and a surface $\{z\}\times\HP$, a {\em vertical plane}.
\end{defn}

At any point in the bidisk, the sectional curvature of any plane is between $-1$
and $0$, and is $-1$ 
if and only if the plane is horizontal or vertical.  In particular, $\BD$ is a
{\em Hadamard manifold}.

We now describe $\isob$, the group of isometries of the bidisk.  Set $\BG=\isoh\times\isoh$. 
 Then $\BG$ acts by isometries on the bidisk via the product action:
\begin{equation*}
(g_1,g_2)(z_1,z_2)=(g_1(z_1),g_2(z_2))
\end{equation*}
where $g_i\in\isoh$ and $z_i\in\HP$.   

Set $\bi$ to be the involution that permutes the coordinates:
\begin{align*}
\bi:\BD & \longrightarrow\BD\\
(z_1,z_2) & \longmapsto (z_2,z_1)
\end{align*}
The map $\bi$ is an isometry which normalizes $\BG$.

\begin{thm}
 The group $\isob$ is $\BG\rtimes\langle\bi\rangle$.
\end{thm}

\begin{proof}
Let $\bg$ be an arbitrary isometry of the bidisk.   If $P$ is a horizontal
plane, then $\bg(P)$ must have sectional curvature $-1$ and thus be a horizontal
or vertical plane. 
 Therefore either $\bg$ or $\bi\circ\bg$ maps horizontal planes to horizontal
planes, 
and vertical planes to vertical planes. Assume without loss of generality that
$\bg$ does.

For any $(x,y)\in\BD$, write $\bg(x,y)=(x',y')$.  Let $x_1,x_2,y\in\HP$ be
arbitrary points.  Since $(x_1,y)$, $(x_2,y)$ belong to the same horizontal
plane, so do their images and thus:
\begin{equation*}
 \HD(x_1,x_2)=\BHD((x_1,y),(x_2,y))=\BHD((x_1',y'),(x_2',y'))=\HD(x_1',x_2')
\end{equation*}
Consequently, the projection of $\bg$ onto the first factor, $x\mapsto x'$, is an
isometry of $\HP$. 

In the same way, considering vertical planes, the projection onto the second
factor is also an isometry of $\HP$.  
Thus $\bg=(g_1,g_2)\in\BG$ (or $\bi\circ\bg\in\BG$) as required.

\end{proof}

One interesting consequence is that while the bidisk is a homogeneous space, it
is not isotropic or equivalently (as it is a Riemannian space), two-point homogeneous.  
Concretely, $\isob$ acts transitively on $\BD$, but no isometry will map, say,
$(x_1,y)$ to itself and $(x_2,y)$ to $(x_2',y')$,
even if the distances are equal, unless $y=y'$.  

\subsection{Flats}

A {\em flat} is a two-dimensional totally geodesic surface that is isometric to the Euclidean
plane.  (Thus it is a {\em maximal flat} in the usual sense, but we will simply call it a flat.)

The only flats in the bidisk are of the form $l_1\times l_2$, where
$l_i\subset\HP$ is a straight line.  As a consequence, one can show that the bidisk is a rank two symmetric space.  
(That the rank is two corresponds to the failure of two-point homogeneity; see for instance~\cite{H}.)

\section{Equidistant hypersurfaces}\label{sec:ehs}

For $\vz,\vw\in\BD$, the {\em equidistant hypersurface} between $\vz $ and $\vw$
is the set of all points $\vx\in\BD$ whose distance from $\vz$ is equal to its
distance from $\vw$:
\begin{equation*}
\E{\vz,\vw} = \{\vx\in\BD~\mid~\BHD(\vx,\vz)=\BHD(\vx,\vw)\}
\end{equation*}
The simplest case is when $\vz=(z_1,z_2)$ and $\vw=(w_1,w_2)$ with either $z_1=w_1$ or
$z_2=w_2$.  Then $\E{\vz,\vw}$ is simply the product of $\HP$ with a straight line.  

Rewrite the equation for an equidistant hypersurface as follows.  Let
$\vz=(z_1,z_2)$ and $\vw=(w_1,w_2)\in\BD$; assume now that $z_i\neq w_i$.  Then $\vx=(x_1,x_2)\in\E{\vz,\vw}$
if and only if $\BHDS(\vx,\vz)=\BHDS(\vx,\vw)$, which is equivalent to :
\begin{equation*}
\HDS(x_1,z_1)-\HDS(x_1,w_1) = \HDS(x_2,w_2)-\HDS(x_2,z_2)
\end{equation*}
This suggests a foliation of the equidistant hypersurface $\E{\vz,\vw}$ by
surfaces which are products of curves in the hyperbolic plane.  

\begin{defn}
Let $k\in\R$, and $z,w\in\HP$ be distinct points.  The {\em hyperbolic square hyperbola}
$\SH{k}{z,w}$ is the curve in the hyperbolic plane defined as follows: 
\begin{equation*}
\SH{k}{z,w}=\{x\in\HP~\mid~\HDS(x,z)-\HDS(x,w)=k\}
\end{equation*}
\end{defn}
Since the context is clear, we will simply call it a {\em square hyperbola}.

Figure~\ref{fig:ThreeSH} shows $\SH{k}{z,w}$ for $z,w$ on the imaginary axis and
$k$ respectively positive, 0 and negative.
\begin{figure}
\includegraphics{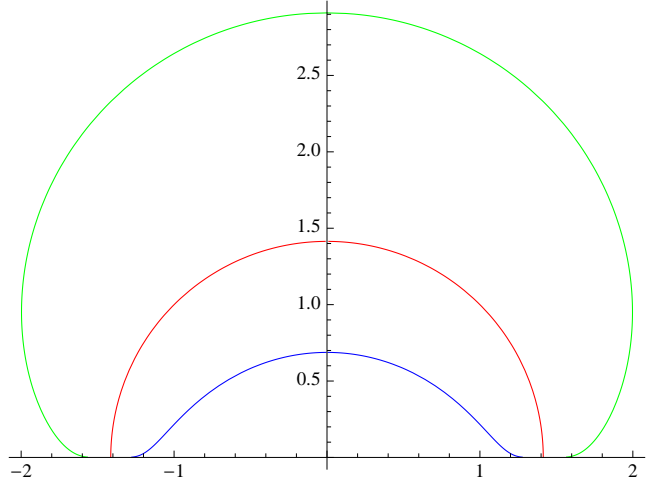}
\caption{Some square hyperbolae $\SH{k}{z,w}$, for $z=I$ and $w=2 I$.  The
middle curve is $\SH{0}{z,w}$, the top curve corresponds to a positive value of
$k$ and the bottom curve, a negative value of $k$.}
\label{fig:ThreeSH}
\end{figure}

Observe that $\SH{k}{z,w}=\SH{-k}{w,z}$.  
Furthermore, $\SH{0}{z,w}$ is simply the equidistant line between $z$ and
$w$.

\begin{lemma}
Let $z,w\in\HP$ be distinct points.  Then for any $k\in\R$, the endpoints of $\SH{k}{z,w}$ are
those of the equidistant line $\SH{0}{z,w}$.
\end{lemma}
\begin{proof}
We can assume, up to the action of an isometry on $\HP$, that $z=aI$, $w=bI$, where $a,b>0$.  
We prove the lemma for $k>0$, the case $k<0$ being analogous.
For $t\geq 0$, set: 
\begin{align*}
C_1(t) & =\{x\in\HP~\mid~\HD(x,aI)=\sqrt{k}\cosh{t}\} \\
C_2(t) & =\{x\in\HP~\mid~\HD(x,bI)=\sqrt{k}\sinh{t}\}
\end{align*}

Then $C_1(t)$ and $C_2(t)$ are hyperbolic circles,
whose Euclidean centers are respectively:
\begin{align*}
c_1(t) & =a\cosh(\sqrt{k}\cosh{t})I \\
c_2(t) & =b\cosh(\sqrt{k}\sinh{t})I
\end{align*}
and whose Euclidean radii are respectively 
\begin{align*}
r_1(t) & =a\sinh(\sqrt{k}\cosh{t}) \\
r_2(t) & =b\sinh(\sqrt{k}\sinh{t})
\end{align*}
There exists $t_0$ such that $C_1(t)\cap C_2(t)\neq\emptyset$ for all $t\geq
t_0$.  
The curve $\SH{k}{z,w}$ consists of all points contained in $C_1(t)\cap C_2(t)$,
$t\geq t_0$. 
Using Euclidean formulas for intersections of circles, we obtain a
parametrization for each branch of $\SH{k}{z,w}$,
according to whether the real part is positive or negative:
\begin{equation}\label{eq:SHparametrization}
t\longmapsto\pm\sqrt{r_1(t)^2-(y(t)-c_1(t))^2}+y(t)I.
\end{equation}
where $y(t)=\frac{b^2-a^2}{2(c_2(t)-c_1(t))}$. 
The limit when $y(t)\rightarrow 0$, that is,
when $t\rightarrow\infty$, is $\pm\sqrt{ab}$, which are precisely the endpoints of $\SH{0}{z,w}$.

\end{proof}

We now introduce notation for surfaces that foliate the equidistant
hypersurfaces.   For $\vz=(z_1,z_2)$, $\vw=(w_1,w_2)\in\BD$, $z_i\neq w_i$, and $k\in\R$, set:
\begin{equation*}
\EK{k}{\vz,\vw}=\SH{k}{z_1,w_1}\times\SH{k}{w_2,z_2}
\end{equation*}
As with square hyperbolae, $\EK{k}{\vz,\vw}=\EK{-k}{\vw,\vz}$.   Thus:

\begin{equation}
\bi\left(\EK{k}{\vz,\vw}\right)  =\EK{k}{\bi(\vw),\bi(\vz)}=
\EK{-k}{\bi(\vz),\bi(\vw)}\label{eq:iotaofE}
\end{equation}

\begin{lemma}
Let $\vz=(z_1,z_2)$ and $\vw=(w_1,w_2)\in\BD$, $z_i\neq w_i$.  Then:
\begin{align*}
\E{\vz,\vw} & = \bigcup_{k\in\R} \EK{k}{\vz,\vw}
\end{align*}
\end{lemma}
The equidistant hypersurface $\E{\vz,\vw}$ contains a single flat, namely, 
$\EK{0}{\vz,\vw}$.  

\begin{defn}
The flat
$\EK{0}{\vz,\vw}$ is called the {\em spine} of $\E{\vz,\vw}$.
\end{defn}

Two distinct equidistant hypersurfaces $\E{\vz,\vw}$ and $\E{\vz',\vw'}$ might
share a common spine, even if $\{\vz,\vw\}\neq\{\vz',\vw'\}$. For instance,
the equidistant surfaces $\E{ ( i/2, i/2 ),( 2i, 2i   ) }$ and 
$\E{ ( i/4, i/4 ),( 4i, 4i   ) }$ share the same spine.
\subsection{Invisibility}
As mentioned in the Introduction, intersections of equidistant hypersurfaces are
difficult to analyze.  
The discussion in \S~\ref{sec:Dirichlet} will require further knowledge of the
location, in the bidisk, 
of even disjoint equidistant hypersurfaces.  To this end, we introduce the
concept of {\em invisibility}.

\begin{defn}
Let $\vx\in\BD$ and $\gamma\in\isob$.  A point $\vy\in\BD$ is {\em $\gamma$-visible to} $\vx$ if:
\begin{enumerate}
\item $\BHD(\vy,\vx)\leq\BHD(\vy,\gamma(\vx))$ and
\item $\BHD(\vy,\vx)\leq\BHD(\vy,\gamma^{-1}(\vx))$.
\end{enumerate}
Otherwise, we say that $\vy$ is {\em $\gamma$-invisible to} $\vx$.
\end{defn}
In other words, $\vy$ being $\gamma$-invisible to $\vx$ means that $\vx$ and $\vy$ are on
opposite sides of 
$\E{\vx,\gamma(\vx)}$ or of $\E{\vx,\gamma^{-1}(\vx)}$.

\begin{defn}
Let $\vx\in\BD$ and $\gamma\in\isob$.  A set $A\subset\BD$ is {\em$\gamma$-invisible to} $\vx$ if every point in $A$ is 
$\gamma$-invisible to $\vx$.
\end{defn}

\begin{lemma}\label{lem:invisibleflats}
Let $\vx,\vy\in\BD$ and $\gamma\in\isob$.  Suppose:
\begin{align*}
\E{\vx,\vy}\cap\E{\vx,\gamma(\vx)} & =\emptyset \\
\E{\vx,\vy}\cap\E{\vx,\gamma^{-1}(\vx)} & =\emptyset
\end{align*}
Then $\E{\vx,\vy}$ is $\gamma$-invisible to $\vx$ if and only if the spine
$\EK{0}{\vx,\vy}$ is $\gamma$-invisible to $\vx$.
\end{lemma}

\begin{proof}
The implication is obvious by the definition of a $\gamma$-invisible set.  
Conversely, suppose $\EK{0}{\vx,\vy}$ is $\gamma$-invisible to $\vx$.  Let $\vw_0\in\EK{0}{\vx,\vy}$; then either 
$\BHD(\vw_0,\vx)>\BHD(\vw_0,\gamma(\vx))$ or $\BHD(\vw_0,\vx)>\BHD(\vw_0,\gamma^{-1}(\vx))$.  Assume that 
$\BHD(\vw_0,\vx)>\BHD(\vw_0,\gamma(\vx))$, since the other case is analogous. Set:
\begin{align*}
 f: \E{\vx,\vy} &\longrightarrow \R \\
\vw & \longmapsto \BHD(\vw,\vx)-\BHD(\vw,\gamma(\vx))
\end{align*}
We claim that $f(\E{\vx,\vy})\subset\R_{>0}$, which implies that $\E{\vx,\vy}$ is $\gamma$-invisible to $\vx$.  
Otherwise, since $\E{\vx,\vy}$ is connected, there exists $\vw\in\E{\vx,\vy}$ such that $f(\vw)=0$.  But then $\vw\in\E{\vx,\gamma(\vx)}$, 
contradicting the assumption that $\E{\vx,\vy}$ and $\E{\vx,\gamma(\vx)}$ are disjoint.
\end{proof}

\begin{rem}\label{rem:invisibleplane}
An analogous notion of invisibility holds in any Riemannian manifold, in particular
the hyperbolic plane.
\end{rem}

\section{Dirichlet domains}\label{sec:Dirichlet}
Let $G$ be a finitely generated discrete group acting on a Riemannian space $X$
with distance function $d$.  Recall that a {\em Dirichlet domain} for $G$ with
basepoint $p\in X$ is the set:
\begin{equation*}
\D{G}{p}=\{q\in X~\mid d(p,q)\leq d(p,g(q))\mbox{ for all }g\neq id\in G\}
\end{equation*}
Specifically, $\D{G}{p}$ is the intersection of all half-spaces $\H_g$, $g\in G$,
where $\H_g$ is the half-space containing $p$ bounded by the locus of equidistant
points between $p$ and $g(p)$.  A Dirichlet domain is a fundamental domain for
the action of $G$ on $X$.  

Let $\gamma\in\Gamma$ and $\vz\in\BD$. Suppose that the equidistant
hypersurfaces $\E{\vz,\gamma(\vz)}$ and $\E{\vz,\gamma^{-1}(\vz)}$ are disjoint.
 Then their spines are disjoint as well. In the hyperbolic plane, if the
equidistant lines between $z$ and $g^j(z)$ are disjoint for $j=\pm 1$, where $g\in\isoh$, then
every equidistant line between $z$ and $g^j(z)$, $j\neq 0,\pm 1$, will be
$g$-invisible to $z$.  (See Remark~\ref{rem:invisibleplane}.)  By
Lemma~\ref{lem:invisibleflats}, $\E{\vz,\gamma(\vz)}$ and
$\E{\vz,\gamma^{-1}(\vz)}$ thus bound a Dirichlet domain for the action of
$\langle\gamma\rangle$.

\begin{thm}\label{thm:main}
Let $\gamma=(g_1,g_2)\in\Gamma$ where both $g_1$ and $g_2$ are hyperbolic.  Let 
$\vz=(z_1,z_2)$ such that $z_i$ lies on the invariant axis for $g_i$, $i=1,2$. 
Then $\E{\vz,\gamma(\vz)}$ and $\E{\vz,\gamma^{-1}(\vz)}$ are disjoint.

In particular, $\D{\Gamma}{\vz}$ is two-faced.
\end{thm}

\begin{proof}
 Suppose $\vx\in\E{\vz,\gamma(\vz)}\cap\E{\vz,\gamma^{-1}(\vz)}$ .  Then the
straight line containing $\vz$, $\gamma(\vz)$ and $\gamma^{-1}(\vz)$
intersects a sphere centered at $\vx$ in three points.  However, $\BD$ being a
Hadamard manifold, 
every sphere bounds a ball that is strictly convex.  
But in that case, straight lines intersect spheres in at most 
two points, contradicting the hypothesis.  
\end{proof}
(The reader unfamiliar with Hadamard manifolds or, more generally, spaces of
non-positive curvature, 
might consult Busemann~\cite{Bu} or Bridson-Haefliger~\cite{BH} for details.)

\subsection{Dirichlet domains with an unbounded number of faces}

In computer experiments, we have obtained examples where $\E{\vz,\gamma(\vz)}$
and $\E{\vz,\gamma^{-1}(\vz)}$ do intersect, when $\vz$ does not belong to the
$\gamma$-invariant flat.  Specifically, taking $g:z\mapsto 2z$ and
$z=1+\frac{1}{4} I$, the square hyperbolae $\SH{10}{g^{-1}(z),z}$ and
$\SH{10}{z,g(z)}$ intersect.  Plausibly, the closer $z$ gets to the
$g$-invariant axis, the larger $k$ will need to be.  It would be interesting to
know whether the Dirichlet domain remains two-faced for points very close to the
invariant flat.  

\begin{conj}
Every $\gamma=(g_1,g_2)$ where both $g_i$'s are hyperbolic admits a Dirichlet
domain with more than two faces.
\end{conj}

While this is an ongoing program, we include here some facts which may prove
useful.

 If $\E{\vz,\vw}$ and $\E{\vz',\vw'}$ intersect, then for some $k,k'\in\R$,
 $\EK{k}{\vz,\vw}$ intersects 
 $\EK{k'}{\vz',\vw'}$.  We would like to understand better the nature of
these intersections and, more specifically, intersections between square
hyperbolae.
 
 \begin{figure}
 \includegraphics{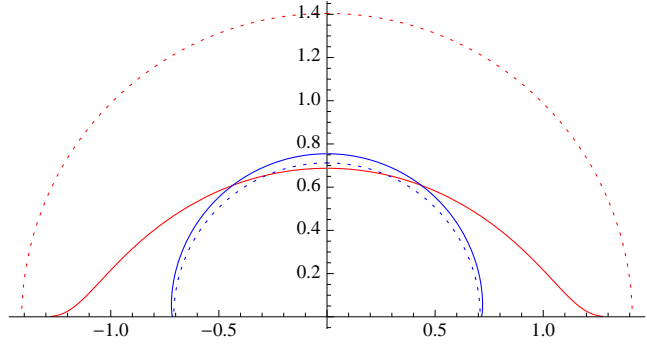}
 \caption{$\SH{k}{ \frac{1}{2} I,I}$ and $\SH{l}{ I,2 I}$, with $k>0$ and $l<0$.
 The dotted lines are the equidistant lines.}
 \label{fig:differentsigns1}
 \end{figure}
 
 \begin{figure}
 \includegraphics{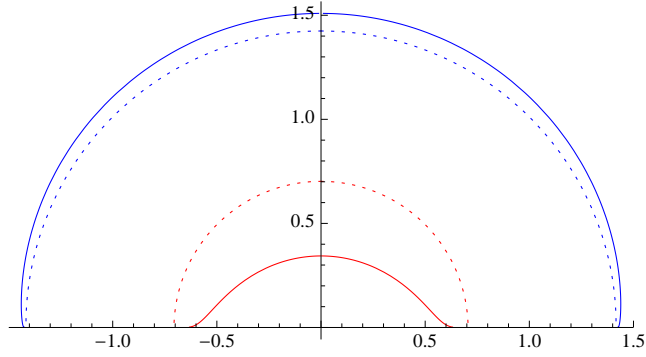}
 \caption{$\SH{k}{ I, 2 I}$ and $\SH{l}{\frac{1}{2} I, I}$, with $k>0$ and
$l<0$.  The dotted lines are the equidistant lines.  The signs of $k,l$ imply
that $\SH{k}{ I, 2 I}$ lies above $\SH{0}{ I, 2 I}$ and $\SH{l}{\frac{1}{2} I,
I}$ lies below $\SH{0}{\frac{1}{2} I, I}$.}
 \label{fig:differentsigns2}
 \end{figure}

We begin with a useful example to keep in mind, in the upper half-plane model.  If $a<b$ are two positive real numbers, the square hyperbola
$\SH{k}{a I,b I}$ lies above the equidistant line $\SH{0}{a I,b I}$ when $k>0$
and below when $k<0$.  In particular, if $a<b<c$ are three positive real numbers, $\SH{l}{a,b}$ and
$\SH{k}{b,c}$ must be disjoint if $k>0$ and $l<0$.  See
Figures~\ref{fig:differentsigns1} and~\ref{fig:differentsigns2}.  
 
 Now suppose: 
 \begin{align*}
 \vx & =(a_1 I,a_2 I) \\
 \vy & =(b_1 I,b_2 I) \\
 \vz & = (c_1 I,c_2 I)
 \end{align*}
 with $0<a_i<b_i<c_i$, $i=1,2$.  If $k>0$ and $l<0$, then by the above discussion
$\SH{l}{a_1 I,b_1 I}$ and $\SH{k}{b_1 I,c_1 I}$ must be disjoint.  If, on the
other hand, $k<0$ and $l>0$, then $\SH{l}{b_2 I,a_2 I}$ and $\SH{k}{c_2 I,b_2
I}$ must be disjoint.  In either case, the following surfaces must be disjoint:
 \begin{align*}
 \EK{l}{\vx,\vy} & =\SH{l}{a_1 I,b_1 I}\times\SH{l}{b_2 I,a_2 I} \\
 \EK{k}{\vy,\vz} &=\SH{k}{b_1 I,c_1 I}\times\SH{k}{c_2 I,b_2 I}
 \end{align*}
 
 We will see now how to generalize this observation.
 
Suppose $l_1,l_2\subset\HP$ is a pair of disjoint straight lines.   (Thus they are
either asymptotic or ultraparallel.)   Since each straight line bounds a half-plane,
the complement of $l_1\cup l_2$ has three components:
\begin{itemize}
 \item two of the components are half-planes, respectively bounded by $l_1$ and $l_2$;
\item the third component is the intersection of the other two half-planes and is bounded by $l_1\cup l_2$.
\end{itemize}
This last component will be called a {\em slab}. 

\begin{defn}
Let $x,y,z\in\HP$; let $l_1$ be the equidistant line between $x$ and $y$ and
$l_2$, the equidistant line between $y$ and $z$.  Suppose $l_1$ and $l_2$ are
disjoint.  We say that $y$ is {\em between} $x$ and $z$ if $y$ belongs to the
slab bounded by $l_1$ and $l_2$.
\end{defn}

For $i=1,2$, denote by $\pi_i$ the projection onto each factor:
\begin{align*}
\pi_i : \BD & \longrightarrow \HP \\
(x_1,x_2) & \longmapsto x_i
\end{align*}

\begin{defn}
 Let $\vx,\vy,\vz\in\BD$.  We say that $\vy$ is {\em between} $\vx$ and $\vz$ if
for each $i=1,2$, $\pi_i(\vy)$ is between $\pi_i(\vx)$ and $\pi_i(\vz)$.
\end{defn}

Note that if $\vy$ is between $\vx$ and $\vz$, then by hypothesis the spines
$\EK{0}{\vx,\vy}$ and $\EK{0}{\vy,\vz}$ are disjoint.

 \begin{lemma}\label{lem:samesign}
 Let $\vx,\vy,\vz\in\BD$ such that $\vy$ is between $\vx$ and $\vz$. Suppose
$\EK{k}{\vx,\vy}$ intersects $\EK{l}{\vy,\vz}$ .  Then $kl>0$.
 \end{lemma}

\begin{rem}
 The order of the points is important: the conclusion is false, for instance,
for $\EK{k}{\vy,\vx}$ and $\EK{l}{\vy,\vz}$.
\end{rem}

\begin{proof}
 Since $\vy$ is between $\vx$ and $\vz$, $kl\neq 0$ by definition.  

Write $\vx=(x_1,x_2),\vy=(y_1,y_2),\vz=(z_1,z_2)$.  Then:
\begin{align*}
 \EK{k}{\vx,\vy} & =\SH{k}{x_1,y_1}\times\SH{k}{y_2,x_2} \\
\EK{l}{\vy,\vz} & =\SH{l}{y_1,z_1}\times\SH{l}{z_2,y_2}
\end{align*}
Consider first, in $\pi_1(\BD)$, the complement of $\SH{0}{x_1,y_1}$ and
$\SH{0}{y_1,z_1}$, respectively the equidistant lines between $x_1$ and $y_1$,
and $y_1$ and $z_1$.  As noted above, two of these components are half-planes
and, $y_1$ being between $x_1$ and $z_1$, one half-plane contains $x_1$ and the
other, $z_1$.  Therefore, the first half-plane contains every square hyperbola
$\SH{k}{x_1,y_1}$ where $k<0$, and the second half-plane contains every
$\SH{l}{y_1,z_1}$ where $l>0$.  Therefore, if $k<0$ and $l>0$, $\EK{k}{\vx,\vy}$
and $\EK{l}{\vy,\vz}$ must be disjoint.

Analogously, in $\pi_2(\BD)$, the complement of the equidistant lines between
$x_2$ and $y_2$, and $y_2$ and $z_2$ consists of three components, two of which
are half-planes containing, respectively, every square hyperbola
$\SH{k}{y_2,x_2}$ where $k>0$ and every $\SH{l}{z_2,y_2}$ where $l<0$.
Therefore, if $k>0$ and $l<0$, $\EK{k}{\vx,\vy}$ and $\EK{l}{\vy,\vz}$ must be
disjoint.

\end{proof}

 \begin{lemma}\label{lem:samevalue}
 Let $\vx=(x_1,x_2)$, $\vy=(y_1,y_2)$ and $\vz=(z_1,z_2)\in\BD$ such that $\vy$
is between $\vx$ and $\vz$ and suppose $\EK{k}{\vx,\vy}$ intersects
$\EK{l}{\vy,\vz}$ .  Then there exists $m\in\R$ such that:
 \begin{itemize}
 \item either $\SH{m}{x_1,y_1}\cap\SH{m}{y_1,z_1}\neq\emptyset$, or
 \item  $\SH{m}{y_2,x_2}\cap\SH{m}{z_2,y_2}\neq\emptyset$.
 \end{itemize}
 \end{lemma}
 \begin{proof}
 By Lemma~\ref{lem:samesign}, $kl>0$.   We may assume without loss of generality
that both $k$ and $l$ are positive.  Indeed, if necessary, consider the images
of the points by $\bi$, using Equation~\eqref{eq:iotaofE} and noting that $\vy$
is in between $\vx$ and $\vz$ if and only if $\bi(\vy)$ is in between $\bi(\vx)$
and $\bi(\vz)$.

{\bf Case 1:} $k<l$.  Denote by $H$ the half-plane bounded by $\SH{0}{x_1,y_1}$
containing $\SH{k}{x_1,y_1}$.  Then $H$ contains $\SH{l}{y_1,z_1}$ as well.  The
complement of $\SH{k}{x_1,y_1}$ in $H$ consists of two components: one
containing all square hyperbolae $\SH{m}{x_1,y_1}$ with $0< m<k$, and one
containing all $\SH{m}{x_1,y_1}$ where $m>k$.  In this second component, each
$\SH{m}{x_1,y_1}$ must intersect $\SH{l}{y_1,z_1}$.  Setting $m=l$ will yield the result.
 
 {\bf Case 2:} $k>l$.   A similar argument holds in $\pi_2(\BD)$ : for $m>l$,
$\SH{m}{z_2,y_2}$ intersects $\SH{l}{y_2,x_2}$.  Setting $m=k$ will yield the result.

 \end{proof}

We can use Lemmas~\ref{lem:samesign} and~\ref{lem:samevalue} to prove Theorem~\ref{thm:main} in a different manner.  For suppose
$\vw\in\E{\vz,\gamma(\vz)}\cap\E{\vz,\gamma^{-1}(\vz)}$, 
where $\gamma=(g_1,g_2)$ is a product of hyperbolic isometries and $\vz=(z_1,z_2)$ belongs to its invariant flat.  
Without loss of generality, let $m\in\R$ and $w\in\HP$ such that:
\begin{equation*}
 w\in\SH{m}{g_1(z_1),z_1}\cap\SH{m}{z_1,g_1^{-1}(z_1)}
\end{equation*}
 Then both $w$ and $g_1(w)$ belong to $\SH{m}{g_1(z_1),z_1}$.  
 This means that there is a ray $y=\kappa x$ (identifying $\HP$ with the upper half of the $(x,y)$-plane) which intersects the square hyperbola in at least two points.
However, a straightforward calculation using the parametrization $x(t)+y(t)I$ in Equation~\eqref{eq:SHparametrization} 
shows that the function $x(t)/y(t)$ is strictly increasing, contradicting the hypothesis.  One can see, at least numerically as in Figure
~\ref{fig:offaxis}, 
that this fails when $\vz$ is no longer assumed to belong to the $\gamma$-invariant flat.

\begin{figure}
 \includegraphics{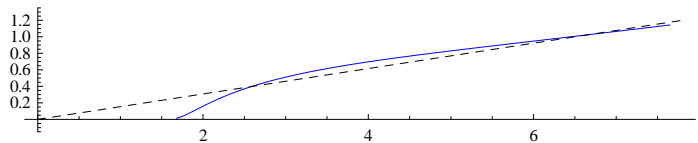}
 \caption{A square hyperbola $\SH{m}{ 1+\frac{1}{4}I, 2+\frac{1}{2}I}$, shown as a thick curve, is intersected in two points by the dashed 
ray 
$y=\kappa x$, with $\kappa$ approximately equal to $0.15$.}
 \label{fig:offaxis}
 \end{figure}


\bibliographystyle{amsplain}
\bibliography{Vref.bib}
 
\end{document}